\documentclass[11pt]{amsart}
\usepackage{amssymb, amsmath}
\title{On varieties where $\mathrm{CS}\mathfrak{X}$ implies $\mathfrak{X}\mathrm{T}$}
\author{O. Al-Raisi}
\address{O. Al-Raisi: Department of  Mathematics,  College of Science, Sultan Qaboos University, Muscat, Oman}
\email{omartalibmiran@gmail.com}

\author{M. Shahryari}
\address{M. Shahryari: Department of  Mathematics,  College of Science, Sultan Qaboos University, Muscat, Oman}
\email{m.ghalehlar@squ.edu.om}

\newtheorem{example}{Example}[section]
\newtheorem{corollary}{Corollary}[section]
\newtheorem{proposition}{Proposition}[section]

\newtheorem {theorem}{Theorem}[section]

\numberwithin{equation}{section}

\newcommand{\XT}{\mathfrak{X}\mathrm{T}}
\newcommand{\CSX}{\mathrm{CS}\mathfrak{X}}
\newcommand{\X}{\mathfrak{X}}
\newcommand{\idX}{\mathrm{id}(\mathfrak{X})}
\newcommand{\idnX}{\mathrm{id}_n(\mathfrak{X})}

\begin{document}

\maketitle
\begin{abstract}
In our previous work \cite{Omar-Shah2}, we initiated the study of $\CSX$- and $\XT$-groups associated with a fixed variety $\X$. A  group belongs to the former class if all of its maximal $\X$-subgroups are malnormal, and to the latter if any two $\X$-subgroups with nontrivial intersection generate an $\X$-subgroup. In general, $\CSX$ does not imply $\XT$, however as shown in \cite{Omar-Shah2}, some varieties do satisfy this implication. In this article, we provide additional examples of varieties for which $\CSX$ implies $\XT$.
\end{abstract}
\vspace{1cm}


%1111111111111111111111111111111111111111111111111111111111111111111111111
%%%%%%%%%%%%%%%%%%%%%%%%%%%%%%%%%%%%%%%%%%%%%%%%%%%%%%%%%%%%%%%%%%%%%%
\vspace{1cm}

In what follows, we use the notations of our previous work \cite{Omar-Shah2}. For an extensive description of the historical motivations, basic concepts, and literature review, we refer the reader to the first few pages of that article. Here, we only review a few basic definitions:

A subgroup $H$ of a given group $G$ is called {\em malnormal} if for every element $x\in G\setminus H$, we have  $H\cap H^x=1$. A group $G$ is called {\em conjugately separable abelian} ($\mathrm{CSA}$) if all maximal abelian subgroups of $G$ are malnormal.
Another class with very close connection to $\mathrm{CSA}$-groups is the class of $\mathrm{CT}$-groups ({\em commutative transitive groups}). A group is $\mathrm{CT}$ if commutativity is a transitive relation on the set of its non-identity elements. It is easy to see that every $\mathrm{CSA}$-group is a $\mathrm{CT}$-group as well. These two fundamental ideas can be generalized for any class of groups (especially for any variety of groups) by considering other laws instead of commutativity. Suppose that $\X$ is a class  of  groups. A group $G$  is an   $\XT$-group  if and only if, for any two $\X$-subgroups $A$ and $B$ of $G$, the assumption $A\cap B\neq 1$ implies that $\langle A, B\rangle$ is also an $\X$-subgroup of $G$. Similarly,  a group $G$ is a $\CSX$-group if  all of its maximal $\X$-subgroups  are malnormal. In \cite{Shah}, the special case where $\X=\mathfrak{N}_c$, the variety of nilpotent groups of nilpotency class not exceeding $c$ was examined. It contains $\mathrm{CSA}$- and $\mathrm{CT}$-groups as particular instances. Hence, \cite{Shah} addresses the study of general relations among the classes of $\CSX$- and $\XT$-groups, their characterizations, constructions, and universal axiomatization, as well as their connections to residual $A$-free groups, in the case when $\X$ is the variety of nilpotent groups of class at most $c$. Furthermore, many previous results are shown to remain valid for $\CSX$ and $\XT$-groups when $\X$ is the variety $\mathfrak{N}_c$. Especially, in this case, it is shown that the property $\CSX$ implies $\XT$, a property which is used in \cite{Omar-Shah} to provide a large classes of equational domains. However, this property happens to fail for other varieties of groups; for example, in \cite{Omar-Shah2} we proved that if $\X$ is the variety generated by  the dihedral group $D_{2p}$ ($p$ is an odd prime), then $\CSX$ is not included in $\XT$. Recognizing all varieties possessing this property might be hard, but as is proved in \cite{Omar-Shah2}, if $\X$ is a variety of nilpotent groups (more generally, any inductive class with nilpotent elements), then $\CSX$ implies $\XT$.

Our aim in this article is to provide more examples of varieties satisfying a {\em weaker version} of the required implication. These examples can be divided into two categories: First,  some varieties which are  quite natural generalizations of nilpotent varieties. They are called varieties of $\X$-nilpotent groups, and we have not found any similar notion in the literature, except for a partial case which had appeared in a Ph.D. thesis in 1971 (see \cite{Teague}). The second set includes all varieties $\X$ satisfying a law of the form
$$
[x, y_1^{p_1}, \ldots, y_k^{p_k}]\approx 1
$$
We show that any torsion-free $\CSX$-group belongs to the class $\XT$. One special case is the variety of $k$-Engel groups, where the result holds even without the restriction of being torsion-free.

For clarity, we review our notations. The subgroup generated by a subset $X$ of a group $G$ will be denoted by $\langle X\rangle$. A conjugate $a^x$ (or $H^x$) is $x^{-1}ax$ (similarly, $x^{-1}Hx$). A commutator $[x, y]$ is $x^{-1}y^{-1}xy$ and all simple commutators $[x_1, x_2, \ldots, x_{k+1}]$ are left aligned. The notation $[x,_k y]$ stands for $[x, y, \ldots, y]$ where $y$ occurs $k$ times.   A variety generated by a single identity $w\approx 1$ is denoted by  $\mathrm{Var}(w\approx 1)$. In most of the cases, we prefer to denote a law $w(x_1, \ldots, x_n)\approx 1$ simply by the element $w$  of the corresponding free group. Hence, whenever we say that $w$ is an identity of the variety $\X$, we mean that all elements of $\X$ satisfy the law $w(x_1, \ldots, x_n)\approx 1$. The set of all ($n$-variable) identities of a variety $\X$ will be denoted by $\idX$ (by $\idnX$).

\section{$\X$-Nilpotent Groups}

In this section, we introduce a natural generalization of nilpotent groups. The proofs of a few statements are  given here as the rest  are more or less similar to the usual arguments about the ordinary nilpotent groups. We begin by recalling a few important definitions. An element of the free group with a countable number of generators $x_1, x_2, x_3, \ldots$ is termed a word. Given a group $G$ and a word $w=w(x_1,\ldots,x_n)$, the verbal subgroup of $G$ corresponding to $w$ is
$$
w(G)=\langle w(g_1,\ldots, g_n): g_1, \ldots, g_n\in G\rangle.
$$
Thus, the subgroup $w(G)$ captures all values of the word $w$ in $G$. It is easy to see that $w(G)$ is a fully invariant subgroup of $G$. Moreover, it is evident that if $\X$ is a variety of groups then $G\in \X$ if and only if $w(G)=1$, for every identity $w$ of $\X$.
Dual to the concept of a verbal subgroup of $G$ corresponding to a word $w=w(x_1,\ldots,x_n)$, is the notion of the $w$-marginal subgroup $w^{\ast}(G)$ of $G$  which is defined as follows: The subgroup $w^{\ast}(G)$ consists of all elements $g\in G$ such that for all $g_1, \dots, g_n\in G$ and all $i\in \{1,\ldots,n\}$,
$$
w(g_1, \ldots, g_{i-1}, gg_i,g_{i+1},\ldots, g_n)=w(g_1, \ldots, g_i, \ldots,g_n).
$$
Two basic but important observations regarding the subgroup $w^{\ast}(G)$ are in order. First, it is easy to see that $w^{\ast}(G)$ is a characteristic subgroup of $G$. Secondly, if $\X$ is a variety of groups then $G\in \X$ if and only if $w^{\ast}(G)=G$, for every identity $w$ in $\X$.

We can define the notion of a verbal subgroup of a group $G$ corresponding to a set of words $W$ (rather than a single word) as
$$
W(G)=\prod_{w\in W} w(G).
$$
Similarly, the marginal subgroup of $G$ corresponding to $W$ is
$$
W^{\ast}(G)=\bigcap_{w\in W}w^*(G).
$$
These will be of particular interest for us when $W$ is the set of defining identities of a variety $\X$ of groups: Let $\X$ be a variety of groups and let $W$ be the set of defining identities of $\X$; let $G$ be a group. We shall write $\X(G)$ and $\X^{\ast}(G)$ for $W(G)$ and $W^{\ast}(G)$, respectively.

The upper $\X$-central series of $G$ is defined inductively: $\X^{\ast}_0(G)=1$ and, for $k\geq 1$, the group $\X^{\ast}_{k}(G)$ is the unique subgroup of $G$ satisfying
$$
\frac{\X^{\ast}_{k}(G)}{\X^{\ast}_{k-1}(G)}=\X^{\ast}(\frac{G}{\X^{\ast}_{k-1}(G)}).
$$
It is easy to see that each subgroup $\X^{\ast}_k(G)$ of $G$ is characteristic in $G$. We say that $G$ is $\X$-nilpotent of class at most $c$ if $\X^{\ast}_{c}(G)=G$ and we use the notation $\mathfrak{N}_c(\X)$ for the class of all $\X$-nilpotent groups of class at most $c$.

A normal series of a group $G$
$$
1=N_0\subseteq N_1\subseteq \cdots \subseteq N_m=G
$$
is said to be an $\X$-series of $G$ if
$$
\frac{N_k}{N_{k-1}}\subseteq \X^{\ast}(\frac{G}{N_{k-1}})
$$
for all $k\in\{1, \ldots, m\}$.\\

We verify that a group $G$ is $\X$-nilpotent if and only if it has an $\X$-series. Also, we  show that the class $\mathfrak{N}_c(\X)$ is a new variety and consequently, most properties of the ordinary nilpotent groups can be generalized to $\X$-nilpotent groups.

\begin{proposition}
A group $G$ is $\X$-nilpotent if and only if it has an $\X$-series.
\end{proposition}

\begin{proof}
By the definition given above, the ascending $\X$-series of $G$ is an $\X$-series when $G$ is $\X$-nilpotent. For the converse, suppose that $G$ admits an $\X$-series:
$$
1=N_0\subseteq N_1\subseteq \cdots\subseteq N_m=G
$$
so that
$$
\frac{N_k}{N_{k-1}}\subseteq \X^*(\frac{G}{N_{k-1}}),
$$
for all $k\in\{1, \ldots,m\}$. We  prove by induction that $N_k\subseteq \X^*_{k}(G)$ for all $k\geq 0$. The inclusion $N_0\subseteq \X^{\ast}_0(G)$ holds trivially since $N_0=1$. Moreover, $N_1=N_1/N_0\subseteq \X^{\ast}(G/N_0)=\X^{\ast}(G)$.
Suppose now that $N_k\subseteq \X^{\ast}_k(G)$ holds for some $k\geq 1$ and let $x\in N_{k+1}$. Since
$$
\frac{N_{k+1}}{N_k}\subseteq \X^{\ast}(\frac{G}{N_k}),
$$
given any $w=w(x_1, \ldots,x_n)\in \idX$, any $i\in\{1, \ldots,n\}$, and any $g_1, \ldots, g_n\in G$, we have
$$
w(g_1, \ldots, g_{i-1}, xg_i, g_{i+1}, \ldots, g_n)N_k=w(g_1, \ldots, g_n)N_k,
$$
therefore,
$$
w(g_1, \ldots, g_n)^{-1}w(g_1, \ldots, g_{i-1}, xg_i, g_{i+1}, \ldots, g_n)\in N_k\subseteq \X^{\ast}_k(G).
$$
Passing to the quotient by $\X^{\ast}_k(G)$ yields
$$
w(g_1, \ldots, g_{i-1}, xg_i, g_{i+1}, \ldots, g_n)\X^{\ast}_k(G)=w(g_1, \ldots, g_n)\X^{\ast}_k(G).
$$
Since the last equality holds for all $w\in \idX$,  $i\in\{1, \ldots, n\}$, and all $g_1, \ldots, g_n\in G$, it follows that
$$
x\X^{\ast}_k(G)\in \X^{\ast}(\frac{G}{\X^{\ast}_k(G)})=\frac{\X^{\ast}_{k+1}(G)}{\X^{\ast}_k(G)},
$$
hence, $x\in \X^{\ast}_{k+1}(G)$. Consequently,  $N_{k+1}\subseteq \X^{\ast}_{k+1}(G)$ which completes the induction, showing that $N_k\subseteq \X^{\ast}_k(G)$, for all $k\geq 0$. In particular, $G=N_m\subseteq \X^{\ast}_m(G)$ from which we conclude that $G$ is $\X$-nilpotent.
\end{proof}

Now, we proceed to define the appropriate generalization of the descending central series.  Given an identity $w=w(x_1, \ldots, x_n)$ and $i\in\{1, \ldots, n\}$, define a new word $w^+_i$ with $n+1$ variables as follows
$$
w^+_i(x_1, \ldots, x_{n+1})=w(x_1, \ldots, x_n)^{-1}w(x_1, \ldots, x_{i-1},x_{n+1}x_i,x_{i+1}, \ldots, x_n).
$$
Furthermore, given subgroups $A$ and $B$ of a group $G$, define $\X^+(A,B)$ to be the subgroup of $G$ generated by $w^+_i(a_1, \ldots,a_n,b)$, where $w\in \idnX$, $i\in \{1, \ldots, n\}$, $a_1,\ldots,a_n\in A$, and $b\in B$. Consequently, it can be shown that a normal series
$$
1=N_0\subseteq N_1\subseteq N_2\subseteq \cdots \subseteq N_c=G
$$
is an $\X$-series if and only if $\X^+(G,N_j)\subseteq N_{j-1}$, for all $j\in\{1, \ldots, m\}$. An important example is the lower $\X$-series defined inductively as
$$
\Gamma^1_{\X}(G)=G, \quad \Gamma^{j+1}_{\X}(G)=\X^+(G, \Gamma^j_{\X}(G)).
$$
Then, a group $G$ is $\X$-nilpotent if and only if $\Gamma^{c+1}_{\X}(G)=1$, for some non-negative integer $c$. The class $\mathfrak{N}_c(\X)$ is indeed a variety as is shown below.

\begin{proposition}
For every variety $\X$, the class $\mathfrak{N}_c(\X)$ of $\X$-nilpotent groups is a new variety containing $\X$. 
\end{proposition}

\begin{proof}
We prove that $\mathfrak{N}_c(\X)$ is closed under subgroups, quotients, and Cartesian products. \\

(i) Let $G\in \mathfrak{N}_c(\X) $ and let $A$ be a subgroup of $G$. We know that $G$ admits an $\X$-series
$$
1=N_0\subseteq N_1\subseteq \cdots \subseteq N_c=G.
$$
Hence, the following series of $A$ is normal as well
$$
1=N_0\cap A\subseteq N_1\cap A\subseteq \cdots \subseteq N_c\cap A=A.
$$
It is enough to show that $\X^+(A, N_j\cap A)\subseteq N_{j-1}\cap A$, for all $j\geq 1$. Note that we already have $\X^+(G, N_j)\subseteq N_{j-1}$, for every $j\geq 1$ and consequently, if $w\in \idnX$, $i\in \{ 1, \ldots, n\}$,  $a_1, \ldots, a_n\in A$ and $y\in N_j\cap A$, then 
\begin{eqnarray*}
w^+_i(a_1, \ldots, a_n, y)&=&w(a_1, \ldots, a_n)^{-1}w(a_1, \ldots, a_{i-1}, ya_i, a_{i+1}, \ldots, a_n)\\
                          &\in& \X^+(G, N_j)\subseteq N_{j-1},
\end{eqnarray*}
and at the same time it belongs to $A$. Therefore, $\X^+(A, N_j\cap A)\subseteq N_{j-1}\cap A$ and $A$ is $\X$-nilpotent of class at most $c$.\\

(ii) Let $G\in \mathfrak{N}_c(\X)$ and consider a normal subgroup  $H\unlhd G$. Again, there exists an $\X$-series of $G$, say
$$
1=N_0\subseteq N_1\subseteq \cdots \subseteq N_c=G,
$$
and consequently, we can consider the series 
$$
1=\frac{N_0H}{H}\subseteq \frac{N_1H}{H}\subseteq \cdots \subseteq \frac{N_cH}{H}=\frac{G}{H}.
$$
Suppose $w\in \idnX$, $i\in \{1, \ldots, n\}$, $a_1H, \ldots, a_nH\in G/H$, and $y\in N_j$. Then 
$$
w_i^+(a_1H, \ldots, a_nH, yH)=w_i^+(a_1, \ldots, a_n, y)H.
$$
But, we know that $w_i^+(a_1, \ldots, a_n, y)\in \X^+(G, N_j)\subseteq N_{j-1}$. So
$$
\X^+(\frac{G}{H}, \frac{N_jH}{H})\subseteq \frac{N_{j-1}H}{H},
$$
showing that $G/H$ belongs to $\mathfrak{N}_c(\X)$. \\

(iii) Finally, let $\{G_{\alpha}:\alpha \in \Lambda\}$ be a family of groups each of which belongs to the class $\mathfrak{N}_c(\X)$; let $G=\prod_{\alpha}G_{\alpha}$ be the  product of this family. Each $G_{\alpha}$ has an $\X$-series of length at most $c$:
$$
1=N_{\alpha}^0\subseteq N_{\alpha}^1\subseteq \cdots \subseteq N_{\alpha}^c=G_{\alpha}.
$$
Now, for each $j\in\{1, \ldots, c\}$, define
$$
M_j=\prod_{\alpha}N_{\alpha}^j
$$
and observe that each $M_j$ is normal in $G$ and that $M_{j-1}\subseteq M_j$, for all $j$. Thus, we have a normal series of $G$:
$$
1=M_0\subseteq M_1\subseteq \cdots \subseteq M_c=G.
$$
It remains to show that this is an $\X$-series of $G$. Again, suppose $w\in \idnX$, $i\in \{1, \ldots, n\}$, $(a_{\alpha}^1)_{\alpha}, \ldots, (a_{\alpha}^n)_{\alpha}\in G$, and $(y_{\alpha})_{\alpha}\in M_j$. Then, we have
$$
w_i^+((a_{\alpha}^1)_{\alpha}, \ldots, (a_{\alpha}^n)_{\alpha}, (y_{\alpha})_{\alpha})=(w_i^+(a_{\alpha}^1, \ldots, a_{\alpha}^n))_{\alpha},
$$
and as $w_i^+(a_{\alpha}^1, \ldots, a_{\alpha}^n)\in N_{\alpha}^{j-1}$, we conclude 
$$
w_i^+((a_{\alpha}^1)_{\alpha}, \ldots, (a_{\alpha}^n)_{\alpha}, (y_{\alpha})_{\alpha})\in M_{j-1}.
$$
\end{proof}

\begin{example}
One may confirm the following using the definition, although the verifications might not be straightforward. 
\begin{enumerate}
\item If $\X=\mathfrak{N}_k$, the variety of nilpotent groups of the class at most $k$, then $\mathfrak{N}_c(\X)=\mathfrak{N}_{ck}$, the variety of nilpotent groups of the class at most $ck$.
\item If $\X=\mathrm{Ab}$ is the variety of all abelian groups, then $\mathfrak{N}_c(\X)=\mathfrak{N}_c$, the variety of ordinary nilpotent groups of class at most $c$. Also, suppose $\X=\mathrm{Ab}_m$, the variety of all abelian groups of exponent $m$. Then $\mathfrak{N}_c(\X)$ is the variety of all groups  $G\in \mathfrak{N}_c$ such that $x^{m^j}\in Z_{c-j}(G)$, for all $j\in \{ 1, \ldots, c\}$. Here $Z_{c-j}(G)$ denotes the $(c-j)$-th term of the ordinary upper central series of $G$.
\item Let $\X$ be the Burnside variety $\mathfrak{B}_m=\mathrm{Var}(x^m\approx 1)$. Define the words $w_1=x_2^{-m}(x_1x_2)^m$ and $w_{j+1}=x_{j+1}^{-m}(w_jx_{j+1})^m$, recursively. Then
$$
\mathfrak{N}_c(\X)=\mathrm{Var}(w_c\approx 1).
$$
\end{enumerate}
\end{example}

 We can develop a suitable theory of $\X$-solvable groups as well. A group $G$ is $\X$-solvable if and only if it has a subnormal series, the quotients in which belong to $\X$. Equivalently,  $G$ is $\X$-solvable if and only if $\mathrm{D}^{d+1}_{\X}(G)=1$, for some $d$, where the $\X$-derived series is defined inductively as
$$
\mathrm{D}^0_{\X}(G)=G, \quad \mathrm{D}^{j+1}_{\X}(G)=\X^+(\mathrm{D}^j_{\X}(G), \mathrm{D}^j_{\X}(G)).
$$
The class of all groups satisfying $\mathrm{D}^{d+1}_{\X}(G)=1$ is a variety containing $\mathfrak{N}_{2^d}(\X)$. We denote this variety by $\mathfrak{A}_d(\X)$.

Most properties of ordinary solvable groups and their relations to nilpotent groups can be proved in this general context.  We close this section by proving three general facts on $\CSX$-groups. In the case of ordinary $\mathrm{CSA}$-groups, these results are well-known (see \cite{Omar-Shah2} and \cite{Shah}).

\begin{proposition}
Let $\X$ be a variety.
\begin{enumerate}
\item Suppose $G$ is a $\CSX$-group. Then any $\X$-solvable subgroup of $G$ belongs to $\X$.
\item If $\X\subseteq \mathfrak{Y}$ where $\mathfrak{Y}$ is a variety contained in $\mathfrak{A}_d(\X)$, then $\CSX\subseteq \mathrm{CS}\mathfrak{Y}$.
\item Let $\overline{\X}=\mathfrak{N}_c(\X)$. Then $\CSX$ implies $\mathrm{CS}\overline{\X}$.
\end{enumerate}
\end{proposition}

\begin{proof}
To prove the first statement, let $H\neq 1$ be an $\X$-solvable subgroup of $G$. As is shown in \cite{Omar-Shah2}, every subgroup of $\CSX$-group is $\CSX$. So $H\in \CSX$. Let $d$ be the smallest number such that $\mathrm{D}^{d+1}_{\X}(H)=1$ and suppose $A=\mathrm{D}^{d}_{\X}(H)$. Then $A$ is a non-trivial normal subgroup of $H$. Note that $\X^+(A, A)=1$ which means that for all $w\in \idnX$, all $i\in \{1, \ldots, n\}$, and any $a_1, \ldots, a_n, b\in A$, we have $w_i^+(a_1, \ldots, a_n, b)=1$ and hence 
$$
w(a_1, a_2, \ldots, a_n)=w(a_1, \ldots, ba_i, \ldots, a_n).
$$
Consequently, repeated application of this fact for $i=1, \ldots, n$ and $b=a_i$, implies 
\begin{eqnarray*}
w(a_1, \ldots, a_n)&=&w(a_11, a_2, \ldots, a_n)\\
                   &=&w(1, a_2, \ldots, a_n)\\
                   &=&w(1, a_21, \ldots, a_n)\\
                   &=&w(1, 1, \ldots, a_n)\\
                   &\vdots&\ \\
                   &=&w(1, 1, \ldots, 1)\\
                   &=&1.
\end{eqnarray*}
This shows that $A\in \X$. Let $M$ be a maximal $\X$-subgroup of $H$ which contains $A$. Then $H$ is not malnormal as for every $h\in H$, we have $A\subseteq M\cap M^h$, except in the case when $M=H$. Therefore, $H\in \X$.

To prove the second statement, suppose that $G$ is $\CSX$ and $H$ is a maximal $\mathfrak{Y}$-subgroup of $G$.  Note that $H$ is $\X$-solvable and therefore, as we just proved, it belongs to $\X$. This implies that $H$ is malnormal, proving that $G\in \mathrm{CS}\mathfrak{Y}$. Finally, applying this result for the special case of $\mathfrak{Y}=\overline{\X}$, we obtain the third statement. 
\end{proof}

\section{Some $\X$-nilpotent varieties satisfying a weak form of $\CSX\to \XT$}
A weaker version of the property of $\X$-transitivity can be defined as follows: Let $\X$ and $\mathfrak{Y}$ be varieties. We say that a group $G$ is $\X$-transitive relative to $\mathfrak{Y}$ if and only if for any two $\X$-subgroups $A$ and $B$ in $G$, the assumption $A\cap B\not\in \mathfrak{Y}$ implies $\langle A, B\rangle\in \X$. We denote the class of all such groups by $\XT_{\mathfrak{Y}}$. As an example, Proposition 3.2 of \cite{Omar-Shah2} shows that if $\X$ is the variety of all metabelian groups and $\mathfrak{Y}$ is the variety of abelian groups, then $\CSX$ implies $\XT_{\mathfrak{Y}}$. In this section, we provide more examples. 

\begin{theorem}
Let $\X$ be a variety which has an identity of the form
$$
[v(x_1, \ldots, x_n), y]\approx 1.
$$
Suppose $\mathfrak{Y}=\mathrm{Var}(v\approx 1)$ and $\overline{\X}=\mathfrak{N}_c(\X)$. Then $\mathrm{CS}\overline{\X}$ implies $\overline{\X}\mathrm{T}_{\mathfrak{Y}}$.
\end{theorem}

\begin{proof}
First, we prove that if a group  $G$ contains two subgroups  $A, B\leq G$  such that $\X^+(A, B)\subseteq A$, then we have  $[v(A), B]\subseteq A$. To this end, suppose $w(x_1, \ldots, x_n, y)=[v(x_1, \ldots, x_n), y]$. The assumption $\X^+(A, B)\subseteq A$ implies that for all $i\in \{1, \ldots, n+1\}$, all $a_1, \ldots, a_{n+1}\in A$, and every $b\in B$, we have
$$
w_i^+(a_1, \ldots, a_{n+1}, b)\in A.
$$
Especially, for the case  $i=n+1$, we obtain
$$
w(a_1, \ldots, a_n, a_{n+1})^{-1}w(a_1, \ldots, a_n, ba_{n+1})\in A
$$
which means that $[v(a_1, \ldots, a_n), ba_{n+1}]\in A$. Consequently,
$$
v(a_1, \ldots, a_n)^{-1}a_{n+1}^{-1}b^{-1}v(a_1, \ldots, a_n)ba_{n+1}\in A
$$
and therefore, $[v(A), B]\subseteq A$.

Now, let $G$ be a $\mathrm{CS}\overline{\X}$-group and $H_1, H_2\leq G$ be two $\overline{\X}$-subgroups with $v(H_1\cap H_2)\neq 1$. We show that $\langle H_1, H_2\rangle \in \overline{\X}$. Let $M$ be a maximal $\overline{\X}$-subgroup of $G$ which contains $H_1$ (such a maximal $\overline{\X}$-subgroup exists by the Zorn's lemma). We prove that $H_2\subseteq M$. Suppose on the contrary that $H_2$ is not included in $M$. Note that as $v(H_1\cap H_2)\neq 1$, we have $H_2\cap M\neq 1$. Also, the assumption $H_2\in \overline{\X}$ implies that
$$
1=\X^{\ast}_0(H_2)\subseteq \X^{\ast}_1(H_2)\subseteq \cdots \subseteq\X^{\ast}_c(H_2)=H_2.
$$
As the first term of this series is included in $M$ but the last one not, there is a $j\in \{ 1, \ldots, c\}$ such that
$$
\X^{\ast}_{j-1}(H_2)\subseteq M, \quad \X^{\ast}_j(H_2)\nsubseteq M.
$$
Recall that $\X^+(H_2, \X^{\ast}_j(H_2))\subseteq \X^{\ast}_{j-1}(H_2)$, so
$$
\X^+(H_2\cap M, \X^{\ast}_j(H_2))\subseteq \X^{\ast}_{j-1}(H_2)\subseteq H_2\cap M.
$$
Hence, by what we have seen above (consider $A=H_2\cap M$ and $B=\X^{\ast}_j(H_2)$), we must have
$$
[v(H_2\cap M), \X^{\ast}_j(H_2)]\subseteq H_2\cap M.
$$
Suppose $z$ is an arbitrary element of $\X^{\ast}_j(H_2)$ which does not belong to $M$. Then $[v(H_2\cap M), z]\subseteq H_2\cap M$. Consequently
$$
1\neq v(H_2\cap M)^z\subseteq H_2\cap M.
$$
On the other hand, we have
$$
1\neq v(H_2\cap M)^z\subseteq (H_2\cap M)^z\subseteq H_2^z\cap M^z,
$$
and this shows that
$$
1\neq v(H_2\cap M)^z\subseteq (H_2\cap M)\cap (H_2^z\cap M^z).
$$
Therefore, $M\cap M^z\neq 1$ and hence, by the malnormality of $M$, we have $z\in M$, a contradiction. Consequently, $H_2$ is included in $M$ and hence $\langle H_1, H_2\rangle \subseteq M$, proving that $\langle H_1, H_2\rangle$ is an $\X$-group. 
\end{proof}

A special case is the variety of center-by-Burnside groups 
$$
\X=\mathrm{Var}([x^m, y]\approx 1).
$$
Suppose that $G$ is a group without nontrivial elements of orders dividing $m$. Applying the above theorem, we conclude that the assumption $G\in \mathrm{CS}\overline{\X}$ implies $G\in \overline{\X}\mathrm{T}$. Note that in the case when $m=1$, the variety $\overline{\X}$ is $\mathfrak{N}_c$, the variety of all nilpotent groups of class at most $c$.  In the next section, we will see that the variety $\X=\mathrm{Var}([x^m, y]\approx 1)$ itself has a similar property; if $G$ does not contain any nontrivial element of order dividing $m$, then the assumption $G\in \CSX$ implies $G\in \XT$.\\

The variety $\X$ in the previous theorem is not the only one with the property $\mathrm{CS}\overline{\X}\to \overline{\X}\mathrm{T}_{\mathfrak{Y}}$. As is shown below, there are more examples.

\begin{theorem}
Let $\X$ be a variety which has an identity of the form
$$
u_1(x_1, \ldots, x_n)y^{-1}v(x_1, \ldots, x_n)yu_2(x_1, \ldots, x_n)\approx 1,
$$
where  $u_1$, $u_2$, $v$ are words in $x_1, \ldots, x_n$ and $y$ is a new variable.
Suppose $\mathfrak{Y}=\mathrm{Var}(v\approx 1)$ and $\overline{\X}=\mathfrak{N}_c(\X)$. Then $\mathrm{CS}\overline{\X}$ implies $\overline{\X}\mathrm{T}_{\mathfrak{Y}}$.
\end{theorem}

\begin{proof}
The proof is completely similar to the previous one, however, we need to modify the first step: Suppose 
$$
w(x_1, \ldots, x_n, y)= u_1(x_1, \ldots, x_n)y^{-1}v(x_1, \ldots, x_n)yu_2(x_1, \ldots, x_n)
$$
and $A$ and $B$ are subgroups of a group $G$ with the property $\X^+(A, B)\subseteq A$. Then again,  $B$ normalizes $v(A)$. To see why, let $a_1, \ldots, a_{n+1}\in A$ and $b\in B$. Then we have
$$
w_{n+1}^+(a_1, \ldots, a_n, ba_{n+1})\in A
$$
and this means that 
$$
u_1(a_1, \ldots, a_n)(ba_{n+1})^{-1}v(a_1, \ldots, a_n)(ba_{n+1})u_2(a_1, \ldots, a_n)\in A.
$$
Consequently, $b^{-1}v(a_1, \ldots, a_n)b\in A$, proving $[v(A), B]\subseteq A$. The rest of the proof goes exactly the same as in the previous one. 
\end{proof}

Especially, this includes the Burnside variety $\mathfrak{B}_m=\mathrm{Var}(x^m\approx 1)$ (simply, take $u_1=u_2=1$ and $v=x^m$). As in the example 1.1, if we define the words $w_1=x_2^{-m}(x_1x_2)^m$ and $w_{j+1}=x_{j+1}^{-m}(w_jx_{j+1})^m$, recursively, then for any $c\geq 1$ the variety $\overline{\X}=\mathrm{Var}(w_c\approx 1)$ has the following property:

\begin{corollary}
Let $\overline{\X}=\mathrm{Var}(w_c\approx 1)$ and $G$ be a $\CSX$-group without any non-trivial element  of order dividing $m$. Then $G$ belongs to $\overline{\X}T$.
\end{corollary}

\section{More examples}
There are some cases where the property $\CSX$ implies $\XT$ but only for a restricted class of groups. A trivial example is the case when $G$ is a finite group (or a locally finite group). As is shown in \cite{Omar-Shah2}, under this extra assumption, the group $G$ belongs to $\X$ and so, it is $\XT$. In this section, we introduce more similar cases.

\begin{theorem}
Let $\X=\mathrm{Var}([x, y_1^{p_1}, \ldots, y_k^{p_k}]\approx 1)$ where $y_1, \ldots, y_k$ are (not necessarily distinct) variables, $x\not\in \{y_1, \ldots, y_k\}$, and the exponents $p_1, \ldots, p_k$ are arbitrary non-zero integers. Let $G$ be a torsion-free $\CSX$-group. Then $G$ belongs to $\XT$.
\end{theorem}

\begin{proof}
Suppose $A$ and $B$ are two $\X$-subgroups of $G$ such that $A\cap B\neq 1$. Let $M$ be a maximal $\X$-subgroup which contains $A$. Suppose $a
\in A\cap B$ is a nontrivial element and $b\in B$. Note that
$$
[b, a^{p_1}, \ldots, a^{p_k}]=1
$$
which is an element of $M$. Hence
$$
[b, a^{p_1}, \ldots, a^{p_{k-1}}]^{-1}a^{-p_k}[b, a^{p_1}, \ldots, a^{p_{k-1}}]a^{p_k}\in M.
$$
As $a\in M$ and $G$ is torsion-free, this implies that
$$
1\neq a^{p_k}\in M\cap M^{[b, a^{p_1}, \ldots, a^{p_{k-1}}]^{-1}}
$$
and consequently, by the malnormality of $M$
$$
[b, a^{p_1}, \ldots, a^{p_{k-1}}]\in M.
$$
Repeating a similar argument, we obtain
$$
[b, a^{p_1}, \ldots, a^{p_{k-2}}]\in M,
$$
and hence, finally $b\in M$. This shows that $B\subseteq M$ and therefore, $\langle A, B\rangle$ belongs to $\X$.
\end{proof}

A special case of the above theorem is the case of the $k$-Engel variety $\X=\mathrm{Var}([x,_k y]\approx 1)$. In this case, we can drop the assumption of being torsion-free.

\begin{theorem}
Let $\X=\mathrm{Var}([x,_k y]\approx 1)$ be the variety of $k$-Engel groups. Then $\CSX$ implies $\XT$.
\end{theorem}

\begin{proof}
Suppose that $G$ is $\CSX$ and $A, B\leq G$ be $\X$-subgroups with a nontrivial intersection. Let $M$ be a maximal $\X$-subgroup which contains $A$ and $a\in A\cap B$ be a nontrivial element. Suppose  $b\in B$ is an arbitrary element. We have  $[b,_k a]=1\in M$ and hence
$$
[b,_{k-1} a]^{-1}a^{-1}[b,_{k-1} a]a\in M
$$
and therefore
$$
1\neq a^{-1}\in M\cap M^{[b,_{k-1} a]^{-1}}.
$$
Then, the malnormality of $M$ implies $[b,_{k-1} a]\in M$ and again repeating the same argument, at the end we obtain $b\in M$.
\end{proof}


\begin{thebibliography}{99}










\bibitem{Teague}
Teague T. K. {\it On marginal subgroups and their generalizations}. Ph.D. thesis, Michigan State University, 1971.

\bibitem{Omar-Shah2}
Al-Raisi O., Shahryari M. {\it On $\mathfrak{X}$-transitive groups and conjugate separable $\mathfrak{X}$-subgroups}. 2026. https://doi.org/10.48550/arXiv.2601.00746

\bibitem{Omar-Shah}
Al-Raisi O., Shahryari M. {\it New classes of groups which are equational domains}. Bull. Iran. Math. Soc. {\bf 51}, Article no. 43, 2025.

\bibitem{Shah}
Shahryari M. {\it On conjugate separability of nilpotent subgroups}. Journal of Group Theory, {\bf 27}(6), 2024, pp. 1171-1185.

\end{thebibliography}
\end{document}